\documentclass[11pt]{article}

\usepackage{amsmath,amsthm,amsfonts} %,dsfont
\usepackage{bm}
\usepackage{natbib}
\usepackage{graphicx,color} 
\usepackage{geometry}% pour gérer les marges du document
\usepackage{calc} 
\title{Inadmissibility of the best equivariant predictive density in the unknown variance case}
\author{A. BOISBUNON and Y. MARUYAMA\\
Center for Spatial Information Science, The University of Tokyo\\ Kashiwa-shi, 277-8568 Chiba, Japan \\
\texttt{aboisbunon@gmail.com} \texttt{maruyama@csis.u-tokyo.ac.jp}}

\date{}

\def\esp{\operatorname{\textit{E}}} % expectation
\def\KLrisk{\ensuremath{R_{KL}}} % KL risk
\def\d{\ensuremath{d}} % Dimension
\def\n{\ensuremath{n}} % Number of observations
\def\s{\ensuremath{s}} % Estimator of the variance
\def\x{\ensuremath{x}} % Observed vectors
\def\xmean{\bar{\x}} % Observed mean
\def\xnn{\xmean_{\n}} % Observed mean
\def\xn1{\xmean_{\n+1}} % Weighted sum of observed mean and future y
\def\snn{\s_{\n}} % Estimator of the variance
\def\sn1{\s_{\n+1}} % Weighted sum of observed mean and future y
\def\y{\ensuremath{y}} % Future data
\def\z{\ensuremath{z}} % Normal vector
\def\Mu{\ensuremath{\mu}} % Mean vector
\def\mN{\ensuremath{N}} % Normal distribution
\def\dbR{\mathbb{R}} % Real space
\newcommand{\idmat}{{I}} % Identity matrix
\def\pred{\hat{p}}  % estimated predictive density
\def\bestEquiv{\pred_R} % Best equivariant predictive density
\def\BayesUs{\pred_{\rm GM}} % Our Bayesian predictive density
\def\constUs{c_{\rm GM}} % Our Bayesian predictive density
\def\ourprior{\pi_{\rm GM}} % Our class of priors
\def\JN{\ensuremath{J}_{\n+1}} %{\RN}} 	% JN
\def\JD{\ensuremath{J}_{\n}} %{\RD}} 	% JD
\def\uN{\ensuremath{u}_{\n+1}} %{\RN}} 	% uN
\def\uD{\ensuremath{u}_{\n}} %{\RD}} 	% uD
\def\priorFunc{\widetilde{\pi}} % Function of lambda in the prior
\def\vmpi{\rho_\pi} % v * marginal density

\theoremstyle{plain}
\newtheorem{theorem}{Theorem}[section]
\newtheorem{lemma}{Lemma}

\begin{document}

%\jname{Biometrika}
%%% The year, volume, and number are determined on publication
%\jyear{}
%\jvol{}
%\jnum{}
%%% The \doi{...} and \accessdate commands are used by the production team
%%\doi{10.1093/biomet/asm023}
%\accessdate{Advance Access publication on }
%\copyrightinfo{\Copyright\ 2012 Biometrika Trust\goodbreak {\em Printed in Great Britain}}
%
%%% These dates are usually set by the production team
%%\received{ }
%%\revised{ }
%
%%% The left and right page headers are defined here:
%\markboth{A. Boisbunon \and Y. Maruyama}{Inadmissibility in the unknown variance case}

%\title{A predictive Stein's effect even in the low-dimensional case}
%\title{Inadmissibility of the best equivariant predictive density in the unknown variance case}
%\author{A. BOISBUNON \and Y. MARUYAMA}
%\affil{Center for Spatial Information Science, The University of Tokyo\\ Kashiwa-shi, 277-8568 Chiba, Japan 
%\email{aboisbunon@gmail.com} \email{maruyama@csis.u-tokyo.ac.jp}}

\maketitle

%\begin{abstract}
%This work treats the problem of estimating the predictive density of a random vector when both the mean vector and the variance are unknown. We prove that the density of reference in this context is inadmissible under the Kullback--Leibler loss in a nonasymptotic framework. Our result holds even when the dimension of the vector is strictly lower than three, which is surprising compared to the known variance setting. 
%Finally, we discuss the relationship between the prediction and the estimation problems.
%\end{abstract}
%
%
%\begin{keywords}
%Bayes rule; inadmissibility; multivariate normal distribution; prior distribution; unknown variance.
%\end{keywords}

%-----------------------------------------------------------------
% 							INTRODUCTION
%-----------------------------------------------------------------

\section{Introduction} \label{sec:intro}

The most natural way to estimate an unobserved quantity $\Mu$ is to use observed averages. %\cite{efron1977stein}. 
However, \cite{stein1955inadmissibility} demonstrated the inadmissibility of such estimators when $\Mu\in\dbR^\d$ with $\d\ge 3$; that is, he showed that there exist better estimators. 
The notion of better estimator means here that it has a lower quadratic risk $R_Q(\Mu,\hat{\Mu})=\esp_{\Mu}(\|\hat{\Mu}-\Mu\|^2)$, where $\hat\Mu$ is an estimator of $\Mu$ and the expectation is taken with respect to random quantities used in the construction of $\hat\Mu$. 
This phenomenon, the Stein effect, was first brought to light in the context of parameter estimation. 
For such a problem, many classes of estimators dominating the average have been proposed. 

In parallel, a similar phenomenon has been observed for predictive density estimation. 
For instance, let us observe $\n$ independent $\d$-dimensional vectors $\x_1,\dots,\x_\n$, each supposedly normally distributed, $\x_i\sim\mN_\d(\Mu,\sigma^2\,\idmat_\d)$ $(i=1,\dots, \n)$, where the common mean $\Mu$ is unknown, and the covariance matrix $\sigma^2\,\idmat_\d$ is assumed known. 
The aim is to estimate the density of a future vector $\y$, which is also assumed to be normal with same mean and variance, $\y\sim\mN_\d(\Mu,\sigma^2\,\idmat_\d)$. 
In this context, the mean $\xmean=\n^{-1}\sum_{i=1}^\n \x_i$ is a sufficient statistic used for the estimation of the predictive density by a function $\pred(\y\mid\xmean)$. 
The quality of such an estimator is often measured by the Kullback--Leibler risk 
\begin{equation*}
  \KLrisk(\Mu,\pred) = \int \phi(\xmean;\Mu,\sigma^2/\n)\!\int\! \phi(\y;\Mu,\sigma^2)\log\frac{\phi(\y;\Mu,\sigma^2)}{\pred(\y\mid\xmean)}d\y  d\xmean,
\end{equation*}
where $\phi(\cdot;\theta,\tau)$ is the multivariate Gaussian density with mean vector $\theta$ and covariance matrix $\tau\idmat_\d$. 
The most natural way to estimate $\phi(\y;\Mu,\sigma^2)$ is to just plug in an estimator of $\Mu$, yielding $\pred(\y\mid\xmean)=\phi\{\y;\Mu=\hat\Mu(\xmean),\sigma^2\}$. 
However, \cite{aitchison1975goodness} proved that plug-in densities are uniformly dominated under the Kullback--Leibler risk by the best equivariant predictive density, which is a Bayesian predictive density taken with respect to the uniform prior $\pi_U(\Mu)=1$, and is given by $\pred_U(\y\mid\xmean) = \phi\{\y;\xmean,(\n+1)\sigma^2/\n\}$. 
\cite{george2006improved} have shown that predictive density estimation and parameter estimation are closely related. 
They argue that the best equivariant predictive density $\pred_U(\y\mid\xmean)$ shares many properties with the maximum likelihood estimator: minimaxity, invariance, constant risk, and inadmissibility for dimension greater or equal to three. 
Hence, it should be taken as the reference estimator. 
\cite{komaki2001shrinkage} showed that, when $\d\ge 3$, it is dominated by the Bayesian predictive density with respect to the harmonic prior $\pi_H(\Mu)= \|\Mu\|^{-(\d-2)}$, also known as the shrinkage prior.

A more interesting and practical setting is when the variance $\sigma^2$ is unknown. 
In such a case, the sufficient statistic is $(\xmean,\s)$,
where $\s=\sum_{i=1}^\n\|\x_i-\xmean\|^2$ estimates the variance up to a factor.
The Kullback--Leibler risk becomes
\begin{equation*} 
  \KLrisk\{(\Mu,\sigma^2),\pred\} = \iint \! \phi(\xmean;\Mu,\sigma^2\!/\n)\frac{\gamma\{\s/\sigma^2;(\n-1)\d\}}{\sigma^{2}}\!\int\! \phi(\y;\Mu,\sigma^2)\log\frac{\phi(\y;\Mu,\sigma^2)}{\pred(\y\mid\xmean,\s)}d\y d\xmean d\s,
\end{equation*}
where $\gamma(\cdot;\upsilon)$ is the chi-square density with  degrees of freedom $\upsilon$. 
The best equivariant predictive density is now taken with respect to the right invariant prior $\pi_R(\Mu,\sigma) = 1/\sigma$ and is the Student $t$-distribution \citep{liang2004exact}
\begin{equation*} %\label{eq:best_equiv}
\bestEquiv(\y\mid\xmean,\s) = c_R
\left(\frac{1}{\pi\s}\,\frac{\n}{\n+1}\right)^{\d/2}\left(1+\frac{1}{\s}\frac{\n}{\n+1}\|\y-\xmean\|^2 \right)^{-\n\d/2},
\end{equation*}
where $c_R = \Gamma(\n\d/2)/\Gamma\{(\n-1)\d/2\}$ and $\Gamma(\cdot)$ denotes the Gamma function. 
So far, there are very few results on admissibility of $\hat{p}_R$. 
Three papers propose extensions to the harmonic prior $\pi_H$, and have influenced the present paper. 
The first is due to \cite{kato2009improved}, who showed that, for $\d\ge 3$ and under the Kullback--Leibler risk, the best equivariant density $\bestEquiv$ is dominated by the Bayesian predictive density with respect to the prior
\begin{equation}\label{eq:kato}
  \pi_{K}(\Mu,\eta)  =  \int_{0}^{1}\phi\left\{\Mu ; 0,\eta^{-1}\lambda^{-1}(1-\lambda)\right\}\pi(\eta,\lambda)d\lambda , \quad \pi(\eta,\lambda)  =  \eta^{-2}\lambda^{-2}1_{(0,1)}(\lambda),
\end{equation}
where $\eta=1/\sigma^2$ and $1_{A}(x)=1$ if $x\in A$ and $0$ otherwise.
\cite{komaki2006shrinkage,komaki2007bayesian} considered a prior based on the Green's function of the manifold of the unknown variance model, which is a hyperbolic plane. 
In an asymptotic framework, he showed the inadmissibility of the best equivariant predictive density under the Kullback--Leibler risk even when $\d<3$.
Finally, \cite{maruyama2012bayesian} uses an extension of the prior $\pi_{K}$ given by equation~\eqref{eq:kato} with
\begin{equation}\label{eq:MS12}
  \pi(\eta,\lambda) = \eta^{a}\lambda^{a}(1-\lambda)^b\,1_{(0,1)}(\lambda),
\end{equation}
where $a$ and $b$ are constants.
This prior leads to an estimator of the density dominating the best equivariant density also in the low dimensional case, where $\d=1$ or 2. 
Unlike \cite{komaki2007bayesian}'s result, the latter one was obtained in a nonasymptotic framework, and under the risk 
\begin{equation*} 
  R_1\{(\Mu,\sigma^2),\pred\} = \iint \! \phi(\xmean;\Mu,\sigma^2\!/\n)\frac{\gamma\{\s/\sigma^2;(\n-1)\d\}}{\sigma^{2}}\!\int\! \pred(\y\mid\xmean,\s)\log\frac{\pred(\y\mid\xmean,\s)}{\phi(\y;\Mu,\sigma^2)}d\y d\xmean d\s,
\end{equation*}
which is denoted by $R_{1}$ because it corresponds to Csisz\'ar's $\alpha$-divergence with $\alpha=1$. 

In the present work, we consider the Kullback--Leibler risk, which is another special case of Csisz\'ar's $\alpha$-divergence with $\alpha=-1.$ 
In the context of unknown variance we show that the inadmissibility of the best equivariant predictive density $\bestEquiv$ for any $\d$, even $\d=1$ or 2, remains true under the Kullback--Leibler risk, hence extending \cite{kato2009improved}'s result. 
Furthermore, we partially solve Problem 2-2 stated by \cite{maruyama2012bayesian}, namely, ``Under $\d=1,2$ and the $R_\alpha$-risk with $-1\le \alpha<1$, does the best invariant predictive density keep inadmissibility? If so, which Bayesian predictive densities improve it?''.
We consider an extension of the harmonic prior different from that of \cite{komaki2007bayesian}, but the major difference lies in the fact that we have nonasymptotic results. Such results are derived under the Gaussian assumption, whereas \cite{komaki2007bayesian} considered a more general distribution.
Finally, we establish a preliminary basis for comparing the estimation and prediction problems for the unknown variance setting, in a similar fashion to \cite{george2006improved}. Such a comparison is however %merely qualitative, and proving it formally is still an open problem. 
much more difficult to prove formally than for the known variance case, thus remaining an open problem.

%-----------------------------------------------------------------------
%			SECTION EXPRESSION OF BAYESIAN PREDICTIVE DENSITY
%-----------------------------------------------------------------------

\section{The Bayesian predictive density under unknown variance} \label{sec:expr}

We are interested in estimating the predictive density of
\begin{equation}\label{eq:gauss_y}
  \y \sim \mN_\d(\Mu, \sigma^{2}\,\idmat_\d),
\end{equation}
based on the observations $\x_1,\dots,\x_\n$, where $\x_i\sim\mN_\d(\Mu,\sigma^2\,\idmat_\d)$.

In this section, we aim at extending the result of \cite{george2006improved} expressing a Bayesian predictive density with respect to a prior $\pi$ as a function of the best equivariant density, i.e.,
\begin{displaymath}
  \hat{p}_\pi(\y\mid\xnn) = \frac{m_\pi\{\xn1;(\n+1)^{-1}\eta^{-1}\}}{m_\pi(\xnn ;\n^{-1}\eta^{-1})}\;\hat{p}_U(\y\mid\xnn), \quad
  m_\pi(\z;\sigma^2) = \int \phi\left(\z;\Mu, \sigma^2\right)\pi(\Mu)d\Mu,
\end{displaymath}
where $\xnn=\xmean$, $\xn1 = (\n \xnn + \y)/(\n+1)$ and $\eta=1/\sigma^2$.
Theorem \ref{th:expr_unknownVar} gives a similar expression for the unknown variance case, relying on both the sufficient statistic $(\xmean,\s)$ and the statistic $(\xn1,\sn1)$, with
\begin{equation*} %\label{eq:imp_stat}
  \xn1 \sim \mN_d\left\{\Mu,(\n+1)^{-1}\eta^{-1}\idmat_\d\right\}, \quad
  \sn1 = \snn + \n\|\y-\xnn\|^2/(\n+1)\sim \eta^{-1}\chi_{\n\d}^2,
\end{equation*}
where $\snn=\s.$
The marginal density $m_\pi(\z,v;l)$ is defined for $\z\sim\mN_\d\{\Mu,(l \eta)^{-1}\idmat_\d\}$ and $v\sim \eta^{-1}\chi_{(l-1)\d}^2$ by
\begin{equation*}
  m_\pi(\z,v;l)  =  \iint \phi\left(\z;\Mu,l^{-1}\eta^{-1}\right) \eta\, \gamma\left\{\eta v;(l-1)\d/2\right\} \pi(\Mu,\eta)d\Mu\, d\eta. %\label{eq:marg}
\end{equation*}
\begin{theorem}\label{th:expr_unknownVar}
  For any prior $\pi(\Mu,\eta)$, the Bayesian predictive density can be expressed as
  \begin{equation*}
    \hat{p}_\pi(\y\mid\xnn,\snn) 
    = \frac{\vmpi(\xn1,\sn1;\n+1)}{\vmpi(\xnn,\snn;\n)}
    \, \hat{p}_R(\y\mid\xnn,\snn),
  \end{equation*}
  where $\hat{p}_R(\y\mid\xnn,\snn)$ is the best equivariant predictive density and $\vmpi= v \, m_\pi(\z,v;l)$.
%   is  the function defined as 
%  \begin{equation}\label{eq:Bayes_best}
%    \vmpi(\z,v;l)= v \, m_\pi(\z,v;l).
%  \end{equation}
  Furthermore, the difference between the Kullback--Leibler risks of  $\hat{p}_\pi$ and $\bestEquiv$ is 
  \begin{equation*}
    R_{KL}\{(\Mu,\eta),\hat{p}_R\}-R_{KL}\{(\Mu,\eta),\hat{p}_\pi\} =   \esp_{(\Mu,\eta)}\!\left\{\log\frac{\vmpi(\xn1,\sn1;\n+1)}{\rho_\pi(\xnn,\snn;\n)} \right\}\! ,
  \end{equation*}
  where $\esp_{(\Mu,\eta)}$ denotes the expectation with respect to $(\y,\xmean,\s)$, provided the expectations exist.
\end{theorem}
Unfortunately, we have not been able to link such an expression directly to the estimation problem, as was done by \cite{george2006improved}, but we believe that the link exists and in Section~\ref{sec:conclu} we give some leads to uncover it.

%-----------------------------------------------------------------------
% 							SECTION MAIN RESULT
%-----------------------------------------------------------------------

\section{Inadmissibility of the best equivariant density} \label{sec:main}

\subsection{Choice of prior}

The easiest way to show the inadmissibility of an estimator is to exhibit a dominating estimator. 
Below, we prove the inadmissibility of the best equivariant density by showing that it is dominated by the Bayesian predictive density with respect to the prior
\begin{equation} \label{eq:us}
  \ourprior(\Mu,\eta) = 
    \int_{0}^{1}\phi\left\{\Mu;0,\eta^{-1}\lambda^{-1}(1-\lambda)\right\} \pi(\eta,\lambda)d\lambda , \quad
    \pi(\eta,\lambda)  =  \eta^{a}\lambda^a\priorFunc(\lambda)\, 1_{(0,1)}(\lambda),
\end{equation}
where the subscript ${\rm GM}$ stands for Gaussian mixture. 
The function $\priorFunc(\lambda)$ will be specified later.
The priors studied by \cite{kato2009improved} and \cite{maruyama2012bayesian} are special cases of this class of priors with $a=-2$ and $\priorFunc(\lambda)=1$ and $\priorFunc(\lambda)=(1-\lambda)^b$. 

We separate two cases, the low-dimensional case where $\d\le 4$, and the higher-dimensional case where $\d\ge 3$. 

%-----------------------------------------------------------------
% 							CASE d<=4
%-----------------------------------------------------------------

\subsection{Domination for the low dimensional case}
\label{sec:low_dim}
Setting $\nu=\d/2+a+1>0 $, we consider the following special case for the prior in \eqref{eq:us},
\begin{equation}\label{eq:choicePi}
  \priorFunc(\lambda) = (1-\lambda)^{(\n-1)\d/2-1} \left\{1-(\n-1)\lambda/\n\right\}^{-(\n-2)\d/2-\nu}, \quad \lambda\in(0,1).
\end{equation}
\begin{theorem}\label{main-thm}
  If $\d\le 4$ and $\n=2$, the best equivariant density $\bestEquiv$ is inadmissible under the Kullback--Leibler risk in the unknown variance setting.
  In particular, it is dominated by the Bayesian predictive density with respect to the prior $\ourprior$ where the function $\priorFunc$ is specified by equation~\eqref{eq:choicePi} with $0<\nu\leq \nu_*$, $\nu_*$ being constant. 
\end{theorem}

\begin{proof} 
The difference in Kullback--Leibler risks between a Bayesian predictive density $\pred_\pi$ and the best equivariant one is 
\begin{equation*} %\label{eq:diff_risk}
  R_{KL}\{(\Mu,\eta),\bestEquiv\}-R_{KL}\{(\Mu,\eta),\hat{p}_\pi\} = \esp_{(\Mu,\eta)}\!\left\{\log \frac{\hat{p}_\pi(\y\mid\xnn,\snn)}{\bestEquiv(\y\mid\xnn,\snn)} \, \right\}\!.
\end{equation*} 
From Lemma~\ref{sec:BayesPredDens}\ref{lem:BayesPredDens}, and for the prior $\ourprior$, where the function $\priorFunc$ is specified by \eqref{eq:choicePi}, this difference in risks is
\begin{align*}
  R_{KL}\{(\Mu,\eta),\bestEquiv\} - R_{KL}\{(\Mu,\eta),\BayesUs\} 
  &  = \esp_{(\Mu,\eta)}\!\bigg[\log \bigg\{c_B
  \left(\frac{\n+1}{\n}\right)^{\nu}  \left(\frac{\sn1}{\snn}\right)^{-\nu}\frac{\JN}{\JD}\bigg\}\bigg]\\
  &  = \esp_{(\Mu,\eta)}\!\left[\log \left\{c_B   \frac{\JN}{\JD} \left(\frac{1-\uD}{1-\uN }\right)^{\nu} \right\}\right]\\
  & \qquad +  \nu \esp_{(\Mu,\eta)}\!\bigg[\log \bigg\{
  \frac{\n+1}{\n} \left(\frac{\sn1}{\snn}\right)^{-1} \frac{1-\uN}{1-\uD}\bigg\}\bigg]\!,
\end{align*}
where $c_B = B\{\nu,(\n-1)\d/2\}/ B(\nu,\n\d/2)$, $u_l=l\|\xmean_l\|^2/(l\|\xmean_l\|^2+\s_l)$ and $J_l$ is given in Equation~\eqref{eq:Jl}, for $l\in\{\n,\n+1\}$.
By Lemmas~\ref{sec:BayesPredDens}\ref{lem:risk_1} and \ref{sec:BayesPredDens}\ref{lem:risk_2}, the risk difference is bounded by
\begin{equation*} 
  R_{KL}\{(\Mu,\eta),\bestEquiv\}- R_{KL}\{(\Mu,\eta),\BayesUs\} \geq \nu \esp( 1-\uN)\left\{g(\n,\d,\nu)-h(\n,\d)\right\}.
\end{equation*}
Since $\nu > 0$ and $1-\uN \ge 0$, the sign of the lower bound is determined by the sign of $\{g(\n,\d,\nu)-h(\n,\d)\}$.
Applying Lemma~\ref{sec:BayesPredDens}\ref{lem:nu_zero} completes the proof.
\end{proof}

Numerical computations give the following approximate values of the constant $\nu_*$: $\nu_*\approx 0.25$ for $\d=1$, $\nu_*\approx 0.33$ for $\d=2$, $\nu_*\approx 0.18$ for $\d=3$, and $\nu_*\approx 0.05$ for $\d=4$. 

Theorem~\ref{main-thm} solves Problem~2.2 stated by \cite{maruyama2012bayesian} for the Kullback--Leibler risk. 
A similar result has been obtained by \cite{komaki2007bayesian}
asymptotically with the number of observations, that is, when $\n\rightarrow\infty$, %whereas 
while here we give it in a nonasymptotic framework for %fixed 
$\n=2$, this restriction coming from Lemma~\ref{sec:BayesPredDens}\ref{lem:nu_zero}. % and $\d$. 
From both results we conjecture that the inadmissibility also stands for $2<\n<\infty$.

%-----------------------------------------------------------------
% 							CASE d>=3
%-----------------------------------------------------------------

\subsection{Domination in higher dimension}

The domination of the Bayesian predictive density $\BayesUs$ over the best equivariant one $\bestEquiv$ does not only occur when  $\d\le 4$. 
Indeed, the following result shows that the domination is still true in higher dimension $\d$, and larger number $\n$ of observations, when this time the function $\priorFunc$ is subject to the following condition, for $0<\nu\le \d/2-1$,
\begin{equation}\label{eq:cond_priorFunc}
  \left\{1-\n\lambda/(\n+1)\right\}^{\d/2-\nu-1}\le \priorFunc(\lambda)\le \left\{1-(\n-1)\lambda/\n\right\}^{\d/2-\nu-1}.
\end{equation}
In this setting, the inadmissibility of the best equivariant predictive density was proved by \cite{kato2009improved}. 
However, only one improving Bayesian predictive density had been provided, whereas we give a class of such dominating densities, including that studied by \cite{kato2009improved}.
\begin{theorem}\label{thm:high-dim}
  If $\d\ge 3$ and $\n\ge 2$, the best equivariant density $\bestEquiv$ is inadmissible under the Kullback--Leibler risk in the unknown variance setting.
  It is dominated by the Bayesian predictive density with respect to the prior $\ourprior$ where $\priorFunc$ is specified by \eqref{eq:cond_priorFunc}. 
\end{theorem}

\begin{proof} %[(Outline)]
Direct calculations show that the lower bound in \eqref{eq:cond_priorFunc} results in the inequalities where
\begin{equation}
  I_l = \int_0^1  t^{\nu-1}\left(1+w_l t\right)^{-(l-1)\d/2-\nu} dt= w_l^{-\nu}  \int_0^{u_l }z^{\nu-1}(1-z)^{(l-1)\d/2-1}dz, \label{eq:upper_JN_bigD}
\end{equation}
the latter equality deriving from the change of variables $z=w_l t/(1+w_l t)$, with $dt = w_l^{-1}(1-z)^{-2}dz$,
where $w_l=l\|\bar{x}_l\|^2/s_l$ and $u_l=w_l/(1+w_l)$.
Gathering those elements, the risk difference can thus be bounded as follows
\begin{equation*}
  \begin{split} %{r l}
    R_{KL}\{(\Mu,\eta),\bestEquiv\}- R_{KL}\{(\Mu,\eta),\BayesUs\}  & \geq  
    \nu\esp_{(\Mu,\eta)}\! \left( \log \left\|\xnn\right\|^{2}- \log\|\xn1\|^2 \right) + r_{\n+1} - r_\n, 
  \end{split}
%  \label{eq:bound_d3}
\end{equation*}
where 
$r_l=\esp\{\log F_{l}(u_l)\}$, $l\in\{\n,\n+1\}$, with 
\begin{equation*}
  F_l(u)=\frac{1}{B\{\nu,(l-1)\d/2\}}\int_0^u t^{\nu-1}(1-t)^{(l-1)\d/2-1}dt, \quad u\in(0,1),
\end{equation*}
and the expectation being taken with respect to $u_l$.
This latter inequality is an equality for the prior $\pi_{K}$ considered by \cite{kato2009improved}. 
Based on \cite{komaki2001shrinkage}, 
we have that
\begin{equation*}
  \esp_{(\Mu,\eta)}\! \!\left( 
    \log \left\|\xnn\right\|^{2}- \log\|\xn1\|^2 \right) \ge 0.
\end{equation*}
Finally, after a few calculations, we obtain that
\begin{equation*}
  r_{\n+1} - r_\n \ge  \esp\!\left[ \log F_{\n+1}\!\left\{ \frac{\chi^2_\d(\n\xi)}{\chi^2_\d(\n\xi)+ \chi^2_{\n\d}} \right\}  \right] - \esp\!\left[ \log F_{\n}\!\left\{ \frac{\chi^2_\d(\n\xi)}{\chi^2_\d(\n\xi)+ \chi^2_{(\n-1)\d}} \right\}\right]\!.
\end{equation*}
From \cite{kato2009improved},
this latter lower bound is positive, 
thus completing the proof.
\end{proof}

The case $\d=\n=2$ makes the link between the two settings. 
Indeed, in such a case, the function $\priorFunc$ defined by equation~\eqref{eq:choicePi} and the upper bound given in equation~\eqref{eq:cond_priorFunc} both equal
$(1-\lambda/2)^{-\nu}$.

Finally, the prior  considered by \cite{maruyama2012bayesian} and given in \eqref{eq:MS12} also belongs to the same class of priors as those specified in the current section, and leads to improvement under the $R_1$-risk. 
This class of priors thus seems to be of special interest and may even include a universal prior leading to improvement for any dimension $\d$, number $\n$ of observations, and any loss. 
However, we have not yet been able to exhibit such a prior.

%-----------------------------------------------------------------
% 					SECTION SIMULATION STUDY
%-----------------------------------------------------------------

\section{Numerical experiment}
\label{sec:data}

In order to check the results of Theorem~\ref{main-thm}, we considered the example given by \cite{kato2009improved}. 
The experiment is driven for  values of the noncentral parameter $\xi=\eta\|\Mu\|^2$ taken in $[0,1\!~000]$, and values of the shape parameter $\nu$ taken in $\{0.05, 0.25, 0.50,\dots , 1.75, 2\}$.
We estimate the difference in risks over $r=5\!~000$ replicates of the different random variables involved and average the results over 10 trials of such an experiment.

Figure \ref{fig:lowdim} shows the behaviour of the estimated difference in risks as a function of $\xi$ and $\nu$ for each $\d\in\{1,2,3\}$ when $\n=2$. 
In this figure, the difference in risks is indeed positive for any $\xi$ and within the range defined by Theorem~\ref{main-thm} for $\nu$, thus confirming our theoretical results. 
What is perhaps more surprising is that the difference in risks is still positive for larger values of $\nu$. 
This is especially noticeable in the case $\d=3$, where the largest improvement seems to be obtained for $\nu\in [1,2]$, whereas in theory we could only prove domination when $\nu\le 0.18$.

\begin{figure}
  \centering
  {\small $\d=1$} \hfil \hfil {\small $\d=2$} \hfil \hfil {\small $\d=3$} \\
  \hspace*{-0.72cm}\includegraphics[width=.35\textwidth]{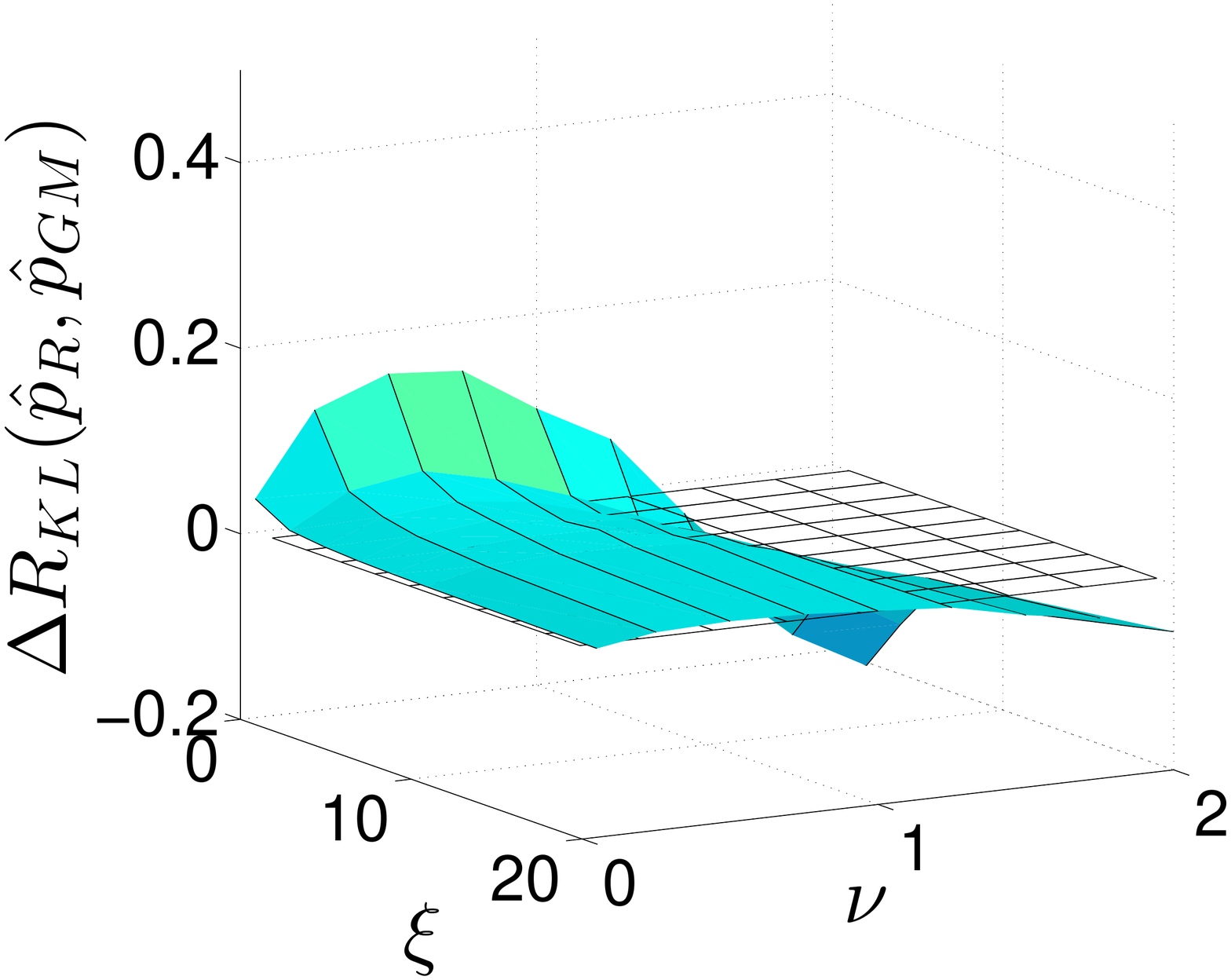}
  \hspace*{-0.78cm}\includegraphics[width=.35\textwidth]{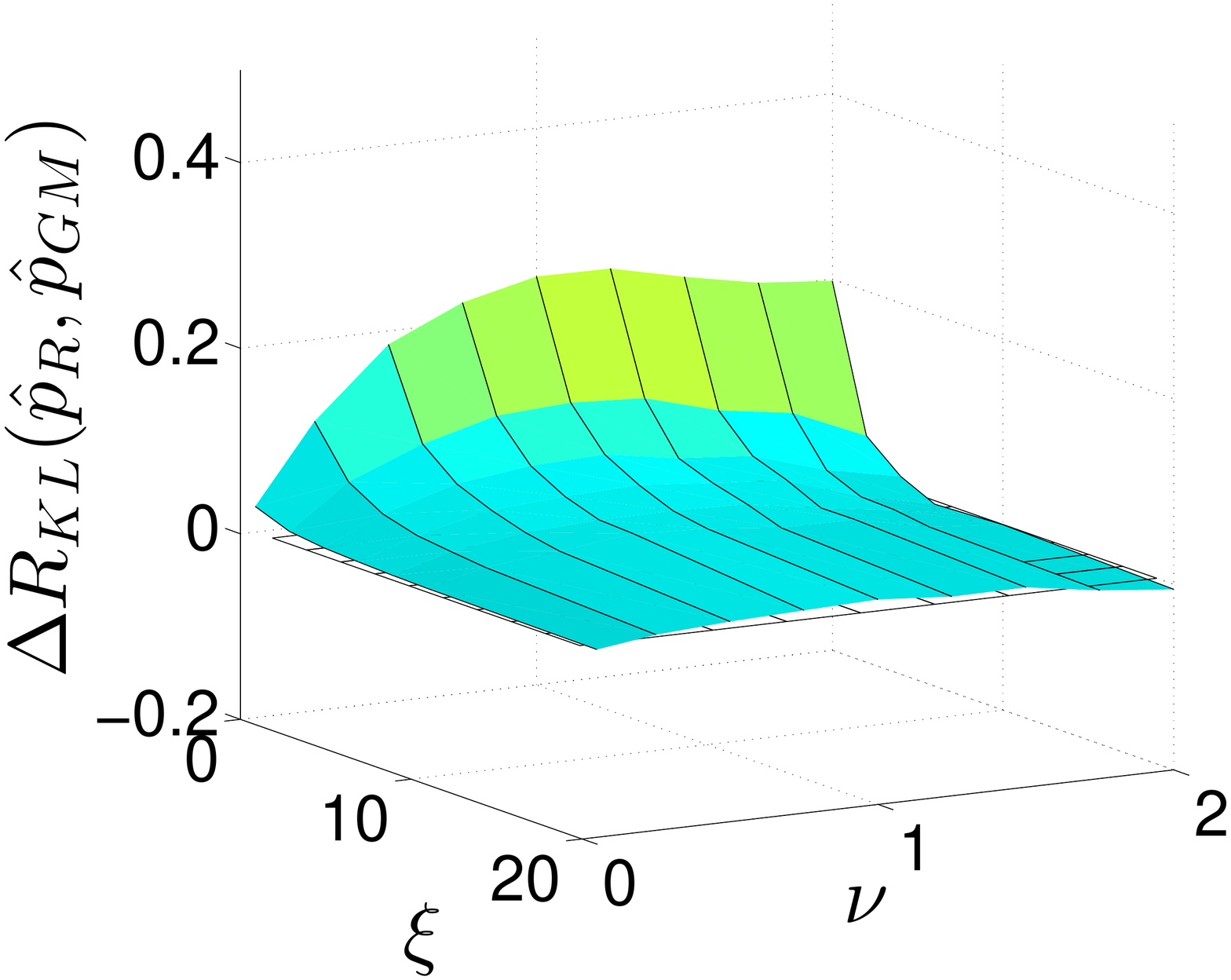}  
  \hspace*{-0.78cm}\includegraphics[width=.35\textwidth]{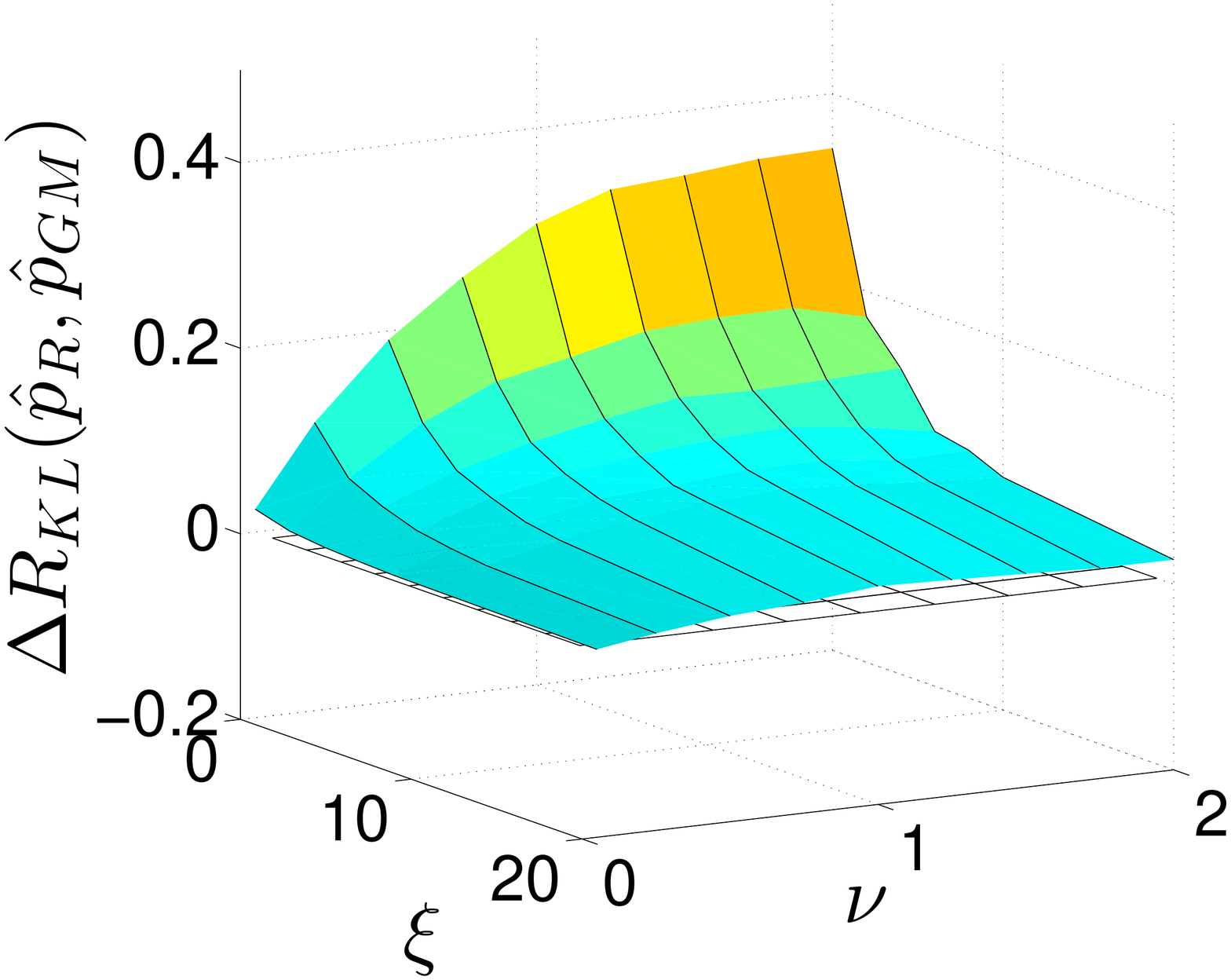}
  \caption{Difference in Kullback--Leibler risks as a function of the noncentrality parameter $\xi=\eta\|\Mu\|^2$ and of the shape parameter $\nu$, for $\d\le 3$ and $\n=2$. The black grid shows the plane $z=0$.}\label{fig:lowdim}
\end{figure}

%-----------------------------------------------------------------
% 							DISCUSSION
%-----------------------------------------------------------------

\section{Discussion}
\label{sec:conclu}

Although we have not been able to prove it so far, it is our conviction that there exists a link between the prediction problem and the estimation problem in the unknown variance case, just as has been shown in the known variance case \citep{george2006improved}.
In order to make this clear, recall that \cite{brown1979heuristic} argues that some decision-theoretic properties do not seem to depend on the choice of the loss function. This claim concerns the estimation problem, but the work of \cite{maruyama2012bayesian}, of \cite{komaki2007bayesian} and the present work, showing inadmissibility of the best equivariant predictive density for any dimension %$\d\ge 1$ 
under two different losses, seem to verify it for the prediction problem as well. Such a relationship is particularly appealing since it was proved by \cite{maruyama2012bayesian} that there exists a direct link between prediction under $R_1$-risk and simultaneous estimation of the mean $\Mu$ and the variance $\sigma^2$. Indeed, the authors show that, for $\pi$ given in equation~\eqref{eq:MS12},
\begin{equation*}
  R_1\{(\Mu,\sigma^2),\hat{p}_\pi\}=
  c_{IQ} \esp_{(\Mu,\sigma^2)}\!\left(\|\hat{\Mu}_\pi-\Mu\|^2/\sigma^2 \right) + c_S \esp_{(\Mu,\sigma^2)}\!\left\{ \hat{\sigma}_\pi^2/\sigma^2-\log(\hat{\sigma}_\pi^2/\sigma^2)-1 \right\}\! ,
\end{equation*}
where $c_{IQ}$ and $c_S$ are constants. 
In this expression, we easily recognise the invariant quadratic risk of $\hat{\Mu}_\pi$ and Stein's risk of $\hat{\sigma}_\pi^2$.
However, this equality follows from the fact that the Bayesian predictive density under the $R_1$-risk  reduces to a plug-in density, which is not the case here. Hence, it is not clear whether the link would still be possible under the Kullback--Leibler risk. 
Another reason for our conjecture is that several arguments given by \cite{george2006improved} are still true under the unknown variance setting: the unbiased estimator of $(\Mu,\sigma^2)$ is equal to the mean and variance of the best equivariant density, and both the unbiased estimators and the best equivariant distribution are minimax and inadmissible whatever the dimension is. 

Theorem~\ref{th:expr_unknownVar} gives an expression of the difference in risks, which can be rewritten as
\begin{equation*}
  R_{KL}\{(\Mu,\sigma^2),\hat{p}_R\}-R_{KL}\{(\Mu,\sigma^2),\hat{p}_\pi\}
  = \int_{\n}^{\n+1} \frac{\partial}{\partial l }\, Q(l,\Mu,\sigma^2)  dl,
\end{equation*}
where $Q(l,\Mu,\sigma^2) = \esp_{(\Mu,\sigma^2)}\{\log \rho_\pi(\z,v;l) \}
$ and $\rho_\pi(\z,v;l)=v\, m_\pi(\z,v;l)$. 
Hence, a sufficient condition for the difference in risks to be positive is that $\partial Q(l,\Mu,\sigma^2)/\partial l \ge 0$ for all $(\Mu,\sigma^2)$ and for $l\in\{\n,\n+1\}$. 
This type of expression is the key element that helped \cite{george2006improved} explicitely express the link with the estimation problem, but its derivation is much more difficult under the unknown variance setting.
It thus remains an open problem to find a similar result when the  variance is unknown.

%-----------------------------------------------------------------
% 			ACKNOWLEDGEMENTS AND SUPPLEMENTARY MATERIAL
%-----------------------------------------------------------------

\section*{Acknowledgement}
This work was supported by Japan Society for the Promotion of Science Postdoctoral Fellowship for Young Foreign Researchers and Japan Society for the Promotion of Science Kakenhi Grant.
% Numbers 23740067 and 25330035.

\section*{Supplementary material}
\label{SM}
The proofs of Theorems~\ref{th:expr_unknownVar} and \ref{thm:high-dim} and of the Lemmas given in the Appendix are provided in Supplementary Material, along with details of the simulation study.

%-----------------------------------------------------------------
% 							APPENDIX
%-----------------------------------------------------------------

\appendix

\section{Lemmas used in the proof of Theorem~\ref{main-thm}}
\label{sec:BayesPredDens}

\begin{lemma}\label{lem:BayesPredDens}
  Under the assumptions of Model \eqref{eq:gauss_y}, the Bayesian predictive density with respect to the prior $\ourprior$ in \eqref{eq:us} is 
  \begin{displaymath}
    \BayesUs(\y\mid\xnn,\snn) = \constUs \left(\frac{\s_{\n+1}}{\s_\n}\right)^{-\nu} \frac{\JN}{\JD}\; \bestEquiv(\y\mid\xnn,\snn) ,
  \end{displaymath}
  where $\constUs=(1+1/\n)^\nu B\left\{\nu,(\n-1)\d/2\right\}/B\left(\nu,\n\d/2\right)$ is a constant and the term $J_l$, $l\in\{\n,\n+1\}$, is defined as
  \begin{equation}
    J_l = \int_0^1 \frac{\lambda^{\nu -1}\{1+(l-1)\lambda\}^{\d/2-\nu-1}}{\left(1+w_l\,\lambda\right)^{(l-1)\d/2+\nu}}
    \,\priorFunc\!\left\{\frac{l\lambda}{1+(l-1)\lambda}\right\}    d\lambda, \label{eq:Jl}
  \end{equation}
  where $\nu=\d/2+a+1$, $B(\cdot,\cdot)$ denotes the Beta function and
  $w_l=l\|\xmean_l\|^2/\s_l$.
\end{lemma}

\begin{lemma}\label{lem:risk_1}
  Let $\n\geq 2$, $0<\nu<1$, and $\priorFunc$ be specified by \eqref{eq:choicePi}. Then 
  \begin{equation*}
    \esp_{(\Mu,\eta)}\!\left(\log \left[\frac{B\{\nu,(\n-1)\d/2\}}{B(\nu,\n\d/2)}    \frac{\JN}{\JD}\frac{(1-\uD)^\nu}{(1-\uN)^\nu}\right] \right)    \geq \nu\, g(\n,\d,\nu) \esp(1-\uN),
  \end{equation*} 
  where $u_l=w_l/(1+w_l)$, $l\in\{\n,\n+1\}$, the expectation $\esp(\cdot)$ is taken with respect to a noncentral Beta random variable, and the function $g$ is defined by
  \begin{equation}\label{eq:gndnu}
    g(\n,\d,\nu)=
    \begin{cases}
    \displaystyle\frac{1}{\nu}\log     \frac{B\{\nu,\n\d/2-(\d-\nu)/\n\}}{B(\nu,\n\d/2)} ,
    & \d\leq \n, \\ 
    \displaystyle{\frac{1}{\n\nu}\,\frac{\d-\nu}{\n\d/2-1}}\log(1+\nu), & \d\geq \n+1.
    \end{cases}
  \end{equation}
\end{lemma}

\begin{lemma}\label{lem:risk_2}
  Let $\n\geq 2$ and $\priorFunc$ be specified by equation~\eqref{eq:choicePi}. Then,
  \begin{equation*}
    \esp_{(\Mu,\eta)}\!\left\{\log\left(\frac{\n+1}{\n}\,     \frac{\snn}{\sn1}\, \frac{1-\uN }{1-\uD}\right) \right\} 
    \geq - h(\n,\d) \esp(1-\uN),
  \end{equation*} 
  where $u_l=w_l/(1+w_l)$, $l\in\{\n,\n+1\}$, the expectation $\esp(\cdot)$ is taken with respect to a noncentral Beta random variable, and the function $h$ is defined by
  \begin{equation}\label{eq:hnd}
    h(\n,\d)=\frac{1+(\n+1)\d/2}{\n\d/2}\left[\psi\{1+(\n+1)\d/2\}-
    \psi(1+\n\d/2)\right],
  \end{equation}
  $\psi(x)=d\log\Gamma(x)/dx$ being the digamma function.
\end{lemma}

\begin{lemma}\label{lem:nu_zero}
  Let $\n=2$ and $1\leq \d \leq 4$. Then, there exists a constant $\nu_*$ depending only on the dimension $\d$ such that, for any $0<\nu\leq \nu_*$,
  \begin{equation*}
    g(\n,\d,\nu) - h(\n,\d)\geq 0.
  \end{equation*}
\end{lemma}

%-----------------------------------------------------------------
% 							REFERENCES
%-----------------------------------------------------------------

\bibliographystyle{apalike}

\begin{thebibliography}{10}
\expandafter\ifx\csname natexlab\endcsname\relax\def\natexlab#1{#1}\fi

\bibitem[{Aitchison(1975)}]{aitchison1975goodness}
\textsc{Aitchison, J.} (1975).
\newblock Goodness of prediction fit.
\newblock \textit{Biometrika} \textbf{62}, 547--554.

\bibitem[{Brown(1979)}]{brown1979heuristic}
\textsc{Brown, L.~D.} (1979).
\newblock A heuristic method for determining admissibility of estimators---with
  applications.
\newblock \textit{Ann. Statist.} \textbf{7}, 960--994.

\bibitem[{George et~al.(2006)George, Liang \& Xu}]{george2006improved}
\textsc{George, E.~I.}, \textsc{Liang, F.} \& \textsc{Xu, X.} (2006).
\newblock Improved minimax predictive densities under {K}ullback--{L}eibler
  loss.
\newblock \textit{Ann. Statist.} \textbf{34}, 78--91.

\bibitem[{Kato(2009)}]{kato2009improved}
\textsc{Kato, K.} (2009).
\newblock Improved prediction for a multivariate normal distribution with
  unknown mean and variance.
\newblock \textit{Ann. Inst. Statist. Math.} \textbf{61}, 531--542.

\bibitem[{Komaki(2001)}]{komaki2001shrinkage}
\textsc{Komaki, F.} (2001).
\newblock A shrinkage predictive distribution for multivariate normal
  observables.
\newblock \textit{Biometrika} \textbf{88}, 859--864.

\bibitem[{Komaki(2006)}]{komaki2006shrinkage}
\textsc{Komaki, F.} (2006).
\newblock Shrinkage priors for {B}ayesian prediction.
\newblock \textit{Ann. Statist.} \textbf{34}, 808--819.

\bibitem[{Komaki(2007)}]{komaki2007bayesian}
\textsc{Komaki, F.} (2007).
\newblock Bayesian prediction based on a class of shrinkage priors for
  location-scale models.
\newblock \textit{Ann. Inst. Statist. Math.} \textbf{59}, 135--146.

\bibitem[{Liang \& Barron(2004)}]{liang2004exact}
\textsc{Liang, F.} \& \textsc{Barron, A.} (2004).
\newblock Exact minimax strategies for predictive density estimation, data
  compression, and model selection.
\newblock \textit{IEEE Trans. Inform. Theory} \textbf{50}, 2708--2726.

\bibitem[{Maruyama \& Strawderman(2012)}]{maruyama2012bayesian}
\textsc{Maruyama, Y.} \& \textsc{Strawderman, W.~E.} (2012).
\newblock Bayesian predictive densities for linear regression models under
  $\alpha$-divergence loss: Some results and open problems.
\newblock In \textit{IMS Collections, Contemporary Developments in {B}ayesian
  Analysis and Statistical Decision Theory: A Festschrift for William E.
  Strawderman}, D.~Fourdrinier, {\'E}.~Marchand \& A.~Rukhin, eds., vol.~8.
  Beachwood, Ohio, USA: Institute of Mathematical Statistics, pp. 42--56.

\bibitem[{Stein(1956)}]{stein1955inadmissibility}
\textsc{Stein, C.} (1956).
\newblock Inadmissibility of the usual estimator for the mean of a multivariate
  normal distribution.
\newblock In \textit{Proceedings of the {T}hird {B}erkeley {S}ymposium on
  {M}athematical {S}tatistics and {P}robability, 1954--1955, vol. {I}}.
  Berkeley and Los Angeles: University of California Press.

\end{thebibliography}

\end{document}